\documentclass{amsart}
\usepackage{amsfonts,amssymb,amsmath,amsthm}
\usepackage{url}
\usepackage{enumerate}
\usepackage[all,cmtip]{xy}

\newtheorem{thrm}{Theorem}[section]
\newtheorem{lem}[thrm]{Lemma}
\newtheorem{prop}[thrm]{Proposition}
\newtheorem{cor}[thrm]{Corollary}
\theoremstyle{definition}
\newtheorem{definition}[thrm]{Definition}

\newtheorem{remark}[thrm]{Remark}
\newtheorem{fact}[thrm]{Fact}

\numberwithin{equation}{section}

\newcommand{\fgrade}{\operatorname{fgrade}}

\newcommand{\Spec}{\operatorname{Spec}}

\newcommand{\Ext}{\operatorname{Ext}}
\newcommand{\Supp}{\operatorname{Supp}}
\newcommand{\End}{\operatorname{End}}

\newcommand{\Char}{\operatorname{Char}}

\newcommand{\Hom}{\operatorname{Hom}}
\newcommand{\Fdepth}{\operatorname{F-depth}}

\newcommand{\Ann}{\operatorname{Ann}}

\newcommand{\depth}{\operatorname{depth}}

\newcommand{\vpl}{\operatornamewithlimits{\varprojlim}}

\newcommand{\vil}{\operatornamewithlimits{\varinjlim}}

\newcommand{\fm}{\frak{m}}
\newcommand{\fD}{\frak{D}}

\newcommand{\fn}{\frak{n}}

\author{ Majid Eghbali}

\thanks{This research was in part supported by a grant from IPM (No. 93130017)}

\address{School of Mathematics, Institute for Research in Fundamental Sciences (IPM), P. O. Box: 19395-5746,
Tehran-Iran.} \email{m.eghbali@yahoo.com}



\begin{document}

\title[Frobenius map, D-modules and local cohomology]{A note on the use of Frobenius map and D-modules in local cohomology
}

\begin{abstract}
The Frobenius depth denoted by $\Fdepth$ defined by
Hartshorne-Speiser in 1977 and later by Lyubeznik in 2006, in a
different way, for rings of positive characteristic. The first aim
of the present paper is to compare the $\Fdepth$ with formal grade
and reprove some results of Lyubeznik using formal local cohomology.
Then the endomorphism rings of local cohomology modules will be
considered. As an application we reprove the results due to
Huneke-Koh in positive characteristic and Lyubeznik in
characteristic zero on the annihilators of local cohomology modules.
\end{abstract}
 \maketitle

\section{Introduction} \label{sect1}

Let $Y$ be a closed subscheme of $\mathbb{P}^n_k$, the projective
space over a field $k$ of characteristic $p>0$. Vanishing of
$H^i(\mathbb{P}^n-Y,\mathcal{F})$ for all coherent sheaves
$\mathcal{F}$ was asked by Grothendieck (\cite{Gr68}). Among the
attempts to answer the mentioned question Hartshorne and Speicer in
\cite{HS} used the notion of Frobenius depth of $Y$ to give an
essentially complete solution to this problem.

To be more precise, Let $Y$ be a Noetherian scheme of finite
dimension, whose local rings are all of characteristic $p>0$. Let $y
\in Y$ be a (not necessarily closed) point. Let $d(y)$ be the
dimension of the closure $\{y\}^-$ of the point $y$. Let
$\mathcal{O}_y$ be the local ring of $y$, let $k_0$ be its residue
field, let $k$ be a perfect closure of $k_0$, and let
$\widehat{\mathcal{O}}_y$, be the completion of $\mathcal{O}_y$.
Choose a field of representatives for $k_0$ in
$\widehat{\mathcal{O}}_y$. Then we can consider
$\widehat{\mathcal{O}}_y$ as a $k_0$-algebra, and we let $A_y$ be
the local ring $\widehat{\mathcal{O}}_y \otimes_{k_0} k$ obtained by
base extension to $k$. Let $Y_y=\Spec A_y$ and let $P$ denote its
closed point. So, the Frobenius depth of $Y$ denoted by $\Fdepth
Y$ is the largest integer $r$ (or $+\infty$) such that for all
points $y \in Y$, one has $H^i_P(Y_y,\mathcal{O}_{Y_y})_s=0$ (the
stable part of $H^i_P(Y_y,\mathcal{O}_{Y_y})$) for all $i < r-d(y)$.

 From the local
algebra point of view, Grothendieck's problem is stated to find
conditions under which  $H^i_I(M)=0$ for all $i>n\ (n \in
\mathbb{Z})$ and all $A$-modules $M$, where $A$ is a commutative
Noetherian local ring and $I \subset A$ is an ideal. For an
$A$-module $M$, we denote by $H^i_I(M)$ the $i$th local cohomology
module of $M$ with respect to $I$. For more details the reader may
consult \cite{Gr67} and \cite{Br-Sh}. From the celebrated result of
Hartshorne (cf. \cite[pp. 413]{Hart68}), it is enough to find
conditions for the vanishing of $H^i_I(A)$. In this direction, for a
local ring $(A,\fm)$, Lyubeznik in \cite{Lyu06} using the Frobenius
map from $ H^i_{\fm}(A)$ to itself defined the Frobenius depth of
$A$ denoted by $\Fdepth A$ as the smallest $i$ such that for every
iteration of Frobenius map, $H^i_{\fm}(A)$ does not go to zero. It
is noteworthy to say that the Lyubeznik's $\Fdepth$ coincides with
the notion of $\Fdepth$ defined by Hartshorne and Speiser, whenever
$A$ admits a surjection from a regular local ring and $Y=\Spec A$
(cf. \cite[Corollary 6.3]{Lyu06}).

Consider the family of local cohomology modules $\{H^i_{\fm}
(M/I^nM)\}_{n \in \mathbb{N}}$ where, $(A,\fm)$ is not necessarily
of characteristic $p>0$. For every $n \in \mathbb{N}$ there is a
natural homomorphism
$$H^i_{\fm} (M/I^{n+1}M) \rightarrow H^i_{\fm} (M/I^nM)$$
(induced from the natural projections $M/I^{n+1}M \rightarrow
M/I^nM$) such that the family forms a projective system. The
projective limit ${\vpl}_nH^i_{\fm}(M/I^n M)$ is called the $i$th
formal local cohomology of $M$ with respect to $I$ (cf. \cite{Sch}).
Formal local cohomology modules were used by Peskine and Szpiro in
\cite{P-S} when $A$ is a regular ring of prime characteristic. It is
noteworthy to mention that if $U = \Spec(A) \setminus \{\fm\}$ and
$( \widehat{U},\mathcal{O}_{\widehat{U}})$ denotes the formal
completion of $U$ along $V(I) \setminus \{\fm \}$ and For an
$A$-module $M$, $\mathcal{F}$ denotes the associated sheaf on $U$
and $\widehat{\mathcal{F}}$ denote the coherent
$\mathcal{O}_{\widehat{U}}$-sheaf associated to ${\vpl}_n M/I^n M$,
they have described the formal cohomology modules
$H^i(\widehat{U},\mathcal{O}_{\widehat{U}})$ via the isomorphisms
$H^i(\widehat{U},\mathcal{O}_{\widehat{U}})\cong
{\vpl}_nH^i_{\fm}(M/I^n M)$, $i \geq 1$. See also \cite[proposition
(2.2)]{Og} when $A$ is a Gorenstein ring.

The formal grade, $\fgrade(I,M)$, is defined as the index of the
minimal nonvanishing formal cohomology module, i.e., $\fgrade(I,M) =
\inf \{i \in \mathbb{Z}|\ \ {\vpl}_nH^i_{\fm}(M/I^n M) \neq 0\}$.
One way to check out vanishing of local cohomology modules is the
following duality
\begin{equation}\label{duality}
    {\vpl}_nH^i_{\fm}(A/I^n) \cong \Hom_A(H^{\dim A-i}_{I}(A),E(A/\fm)),
\end{equation}
where $(A,\fm)$ is a Gorenstein local ring and $E(A/\fm)$ denotes
the injective hull of the residue field (cf. \cite[Remark
3.6]{Sch}). To be more precise, in this case the last non vanishing
amount of $H^{i}_{I}(A)$ may be described with the $\fgrade(I,A)$.
Thus, it motivates us to consider the invariants $\Fdepth$ and
$\fgrade$ to shed more light on the notion of Frobenius depth from a
different point of view. For this reason, in Section $2$, we bring
some auxiliary results and among them we examine the structure of
${\vpl}_nH^i_{\fm}(A/I^n)$ as a unite $A[F^e]$-module (Theorem
\ref{unit}). In Section $3$, we show that the $F$-nilpotency of
$H^i_{\fm}(A/I),\ (i \in \mathbb{Z})$ is equivalent to the vanishing
of ${\vpl}_nH^i_{\fm}(A/I^n)$ for a regular local $F$-finite ring
$A$ (Proposition \ref{f-nil}). It helps us to reprove
\cite[Corollary 3.2, Lemma 2.4]{Lyu06}, see Corollary \ref{Lyu3.2}
and Proposition \ref{lyu4.2}. Finally, we compare the Frobenius
depth with the formal grade and depth, (cf. Theorem \ref{comparison}
and Corollary \ref{connected}).

Studying of the endomorphism rings of local cohomology modules are
of interest in recent years. See for instance Hochster-Huneke
\cite{Ho-Hu}, Hellus-St\"{u}ckrad \cite{He-S}, Schenzel
\cite{Sch09}, \cite{Sch10}, Eghbali-Schenzel \cite{E-Sch},
Mahmood-Schenzel \cite{Ma-Sch}. In Section $4$, we examine the
endomorphism rings of local cohomology modules,
$\Hom_R(H^i_I(R),H^i_I(R))$, where $R$ is a regular local ring
containing a field in both characteristic zero and $p>0$. Our way to
prove the results is to use the so-called $\fD$-modules and
Frobenius map. These methods have played decisive roles in many
subsequent studies in the rings of characteristic zero and positive
characteristic. As an application we reprove the results due to
Huneke-Koh \cite[Lemma 2.2]{Hu-K} in positive characteristic and
Lyubeznik \cite[Corollary 3.6]{Lyu93} in characteristic zero on the
annihilators of local cohomology modules (cf. Corollaries
\ref{annloc.0} and \ref{annloc.p}).

\section{Auxiliary Results in positive characteristic}

Throughout this section all rings are assumed to contain a field of
positive characteristic. The symbol $A$ will always denote a
commutative Noetherian ring of finite characteristic. We adapt the
notation from \cite{Bl} and except for notation we mostly follow
Lyubeznik \cite{Lyu97}. We let $F=F_A$ the Frobenius map on $A$,
that is $F:A \rightarrow A$, with $a \mapsto a^p,\ a \in A$. We
denote the $e$th iterate of the Frobenius map by $A^e$ which is the
$A-A$-bimodule. As a left $A$-module it is $A$ and as a right
$A$-module we have $m.a=a^{p^e}m$ for $m \in A^e$. We say $A$ is
F-finite, whenever $A^e$ is a finitely generated right $A$-module.

\begin{remark}\label{Hart-Speis}
Let us recall from \cite[Proposition 1.1(a)]{HS} that for a ring $A$
which is either a localization of an algebra of finite type over a
perfect field $k$, or a complete local ring containing a perfect
field $k$ as its residue field, then $A$ is $F$-finite.
\end{remark}

In the present section, among our results we recall various results
due to Hartshorne-Speiser \cite{HS}, Peskine-Szpiro \cite{P-S},
Lyubeznik \cite{Lyu97} and Blickle \cite{Bl}.

Peskine and Szpiro in \cite{P-S} defined the Frobenius functor as
follows:

\begin{definition} The Frobenius functor is the right exact functor
from $A$-modules to $A$-modules given by
$$F^{\ast}_A M:=A^1 \otimes_A M.$$
Its $e$th power is $F^{e\ast}_A M=A^e \otimes_A M$. For brevity we
often write $F^{e\ast}$ for $F^{e\ast}_A $ when there is no
ambiguity about the ring $A$.
\end{definition}

It follows from the definition that $F^{e\ast}$ commutes with direct
sum, direct limit and localization. By a theorem of Kunz \cite{Ku69}
the Frobenius functor is flat whenever $A$ is a regular ring, hence
in this case $F^{e\ast}$ will be exact and immediately one has
$F^{e\ast} A=A^e \otimes_A A \cong A$ and $F^{e\ast} I=I^{[p^e]}$ an
ideal of $A$ generated by $p^e$th powers of the elements of $I$.

\begin{definition} \label{def.unit} An $A[F^e]$-module is an $A$-module $M$ together with an $A$-linear map
$$\mathcal{V}^e_M:F^{e\ast}_AM \rightarrow M.$$
A morphism between two $A[F^e]$-modules $(M,\mathcal{V}^e_M)$ and
$(N,\mathcal{V}^e_N)$ is an $A$-linear map $\varphi : M \rightarrow
N$ such that the following diagram commutes:

$$\xymatrix{
F^{e\ast}M \ar[d]^{\mathcal{V}^e_M} \ar[r]^{F^{e\ast}(\varphi)} &F^{e\ast}N\ar[d]^{\mathcal{V}^e_N}\\
M \ar[r]^{\varphi} &N}.$$ In fact, we can consider $F^eM$ as a
$p^e$-linear map from $M \rightarrow M$; as such it is not
$A$-linear but we have $F^e(am) = a^{p^e}F^e(m),\ a \in A^e, m \in
M$. Furthermore, an $A[F^e]$-module $(M,\mathcal{V}^e_M)$ is called
a unite $A[F^e]$-module if $\mathcal{V}^e_M$ is an isomorphism (cf.
\cite[page 17, Definition 2.6]{Bl}).
\end{definition}

\begin{remark}
Since $F^{e\ast}A$ is canonically isomorphic to $A$, then $A$ is a
unite $A[F^e]$-module, but an ideal $I \subset A$ is not a unite
$A[F^e]$-module in general, as the inclusion $I^{[p^e]} \subseteq I$
can be strict. For a multiplicatively closed subset $S$ of $A$, the
structural map $\mathcal{V}^e_{S^{-1}A}:A^e \otimes S^{-1}A
\rightarrow S^{-1}A$ is an isomorphism (see \cite{Bl}).
\end{remark}

The following definition is introduced in \cite{HS}:

\begin{definition} \label{forgetful} Let $(M,\mathcal{V}^e)$ be an $A[F^e]$-module.
We define $G(M)$ as the inverse limit generated by the structural
map $\mathcal{V}^e$, i.e.
$$G(M):={\vpl}(\cdots \rightarrow F^{3e\ast}M \stackrel{F^{2e\ast}\mathcal{V}^e}{\longrightarrow}
 F^{2e\ast}M \stackrel{F^{e\ast}\mathcal{V}^e}{\longrightarrow} F^{e\ast}M \stackrel{\mathcal{V}^e}{\longrightarrow} M).$$
\end{definition}
Note that there are natural maps $\pi_e: G(M) \rightarrow
F^{e\ast}M$. Moreover, the maps $F^{e\ast}\pi_r: F^{e\ast}G(M)
\rightarrow F^{e(r+1)\ast}M$ are compatible with the maps defining
$G(M)$ and thus by the universal property of inverse limits, the map
$F^{e\ast}G(M) \rightarrow G(M)$ defines the natural $A[F^e]$-module
structure on $G(M)$.

\begin{prop}  \label{4.1}(\cite[Proposition 1.2]{HS} and \cite[Proposition 4.1]{Bl}) Let $A$ be regular and $F$-finite and $M$ an
$A[F^e]$-module. Then $G(M)$ is a unite $A[F^e]$-module.
\end{prop}

In order to extend the Matlis duality functor $D(-)=\Hom(-, E_A)$
where, $E_A$ is the injective hull of the residue field, the functor
$\mathcal{D}$ from $A[F^e]$-modules to $A[F^e]$-modules is defined
as follows ( \cite[Section 4]{Bl}):

Let $(M, \mathcal{V}^e)$ be an $A[F^e]$-module. We define
$$\mathcal{D}(M)={\vil}(D(M) \stackrel{D(\mathcal{V}^e)}{\longrightarrow}
 D(F^{e\ast}M) \stackrel{D(F^{e\ast}(\mathcal{V})}{\longrightarrow} D(F^{2e\ast}M) \longrightarrow \ldots).$$
An element $m \in M$ is called $F$-nilpotent if $F^{re}(m)=0$ for
some $r$. Then $M$ is called $F$-nilpotent if $F^{er}(M)=0$ for some
$r \geq 0$. It is possible that every element of $M$ is
$F$-nilpotent but $M$ itself is not.

Below, we recall some properties of the functor $\mathcal{D}$.

\begin{prop} \label{4.16-4.17}Let $A$ be a complete regular local ring.

(a) On the subcategory of $A[F^e]$-modules which are cofinite (i.e.
satisfy the descending chain condition for submodules) as
$A$-modules,  the functor $\mathcal{D}$ is exact and
$\mathcal{D}(N)$ is a finitely generated unite $A[F^e]$-module for
every cofinite module $N$ (cf. \cite[Theorem 4.2(i)]{Lyu97} and
\cite[Proposition 4.16]{Bl}).

(b) Let $M$ be an $A[F^e]$-module that is finitely generated or
cofinite as an $A$-module. Then $\mathcal{D}(\mathcal{D}(M))\cong
G(M)$ (cf. \cite[Proposition 4.17]{Bl}).

(c) Let $M$ be an $A[F^e]$-module which is a cofinite $A$-module.
Then $M$ is $F$-nilpotent if and only if $\mathcal{D}(M) = 0$ (cf.
\cite[Theorem 4.2(ii)]{Lyu97} and \cite[Proposition 4.20]{Bl}).
\end{prop}

Let $A$ be a regular local ring and let $I$ be an ideal of $A$. As
we have seen (following \cite[page 19]{Bl}) that $A$ is an
$A[F^e]$-module and this structure comes down to its localization.
The local cohomology modules $H^i_I(A)$ of $A$ with support in $I$
can be calculated as the cohomology modules of the \v{C}ech complex
$$\check{C}(A; x_1,\ldots, x_n) = A \rightarrow A_{x_i} \rightarrow A_{x_ix_j} \rightarrow A_{x_1x_2\cdots x_n}$$
where, $I$ is generated by $x_1,x_2,\ldots, x_n$. As the category of
unite $A[F^e]$-modules is an abelian subcategory of the category of
$A[F^e]$-modules, the modules $H^i_I(A)$ are unite $A[F^e]$-modules
for all $i \in \mathbb{Z}$. For formal local cohomology modules, the
situation is a bit more complicated, however, we show that these
kind of modules have unite $A[F^e]$-modules structure, where $A$ is
$F$-finite.

\begin{thrm} \label{unit} Let $(A,\fm)$ be a
regular $F$-finite local ring. Then
$$G(H^i_{\fm}(A/I)) \cong
{\vpl}_nH^i_{\fm}(A/I^{n}),\ i \in \mathbb{Z}$$ which is a unite
$A[F^e]$-module. In particular, $$ {\vpl}_nH^i_{\fm}(A/I^{n}) \cong
H^i_{\fm}(\hat{A^I}),\ i \in \mathbb{Z}$$ as $A[F^e]$-modules, where
$\hat{A^I}$ is the completion of $A$ along $I$.
%the endomorphism ring of
%${\vpl}_nH^i_{\fm}(A/I^{n})$, i.e.
%$\End_R({\vpl}_nH^i_{\fm}(A/I^{n}))$ is a unit $A[F^e]$-module for
%all $i$.
\end{thrm}

\begin{proof}
By what we have seen above,  $H^i_{\fm}(A/I)$ is an $A[F^e]$-module
for all $i \in \mathbb{N}$. Now, we may apply functor $G(-)$ to
$H^i_{\fm}(A/I)$:
$$G(H^i_{\fm}(A/I))=
{\vpl}(\cdots \stackrel{F^{2e\ast}\mathcal{V}^e}{\longrightarrow}
 H^i_{\fm}(A/I^{[p^{2e}]}) \stackrel{F^{e\ast}\mathcal{V}^e}{\longrightarrow} H^i_{\fm}(A/I^{[p^{e}]}) \stackrel{\mathcal{V}^e}
 {\longrightarrow} H^i_{\fm}(A/I)).$$
By virtue of Proposition \ref{4.1}, $G(H^i_{\fm}(A/I))$ is a unite
$A[F^e]$-module. Notice that, the right hand side is nothing but
${\vpl}_eH^i_{\fm}(A/I^{[p^e]})$.

On the other hand, one has ${\vpl}_nH^i_{\fm}(A/I^{n}) \cong
\Hom_A(H^{\dim A-i}_{I}(A),E)$, where $E:=E(A/\fm)$ is the injective
hull of the residue field. The natural map
$$F^{e \ast} \Hom_A(H^{\dim A-i}_{I}(A),E) \rightarrow \Hom_A(H^{\dim A-i}_{I}(A),E)$$
by sending $r \otimes \varphi$ to $rF^{e \ast}(\varphi)$ is an
isomorphism of $A[F^e]$-modules ($r \in A$ and $\varphi \in
\Hom_A(H^{\dim A-i}_{I}(A),E)$). To this end note that $H^{\dim
A-i}_{I}(A)$ and $E \cong H^{\dim A}_{\fm}(A)$ carry natural unite
$A[F^e]$-structure. Thus, ${\vpl}_nH^i_{\fm}(A/I^{n})$ is a unite
$A[F^e]$-module for each $i \in \mathbb{Z}$.

In order to complete the proof, it is enough to show that
${\vpl}_eH^i_{\fm}(A/I^{[p^e]}) \cong {\vpl}_nH^i_{\fm}(A/I^{n})$.
For this reason, consider the decreasing family of ideals
$\{I^{[p^e]}\}_e$. Clearly, its topology is equivalent to the
$I$-adic topology on $A$. Thus, by \cite[Lemma 3.8]{Sch} there
exists a natural isomorphism
$${\vpl}_eH^i_{\fm}(A/I^{[p^e]}) \cong
{\vpl}_nH^i_{\fm}(A/I^{n})$$ for all $i \in \mathbb{Z}$.

The last part follows by \cite[Proposition 2.1]{HS}.
\end{proof}

\section{Frobenius depth} \label{sect2}

Let $A$ be a regular local $F$-finite ring of characteristic $p>0$
and let $I$ be an ideal of $A$. As we have seen in the previous
section the formal local cohomology modules, are unite
$A[F^e]$-modules. In the present section we use the unite $A[F^e]$
structure of ${\vpl}_nH^i_{\fm}(A/I^{n})$ in order to prove our
results.

\begin{prop} \label{f-nil} Let $(A,\fm)$ be a
regular local and $F$-finite ring. Then
${\vpl}_nH^i_{\fm}(A/I^{n})=0$ if and only if $H^i_{\fm}(A/I)$ is
$F$-nilpotent.
\end{prop}

\begin{proof}
% (see proof of \cite[Proposition 4.16]{Bl})
By the assumptions $A$ is $F$-finite that is $A^e$ is a finitely
generated $A$-module. Then tensoring with $A^e$ commutes with the
inverse limit, as $A^e$ is a free right $A$-module (cf.
\cite[Proposition 1.1(b)]{HS}). Thus, we have
$$A^e \otimes_A \widehat{A}=A^e \otimes_A ({\vpl}_n A/I^n) \cong {\vpl}_n A^e/A^eI^n= {\vpl}_n A^e/I^{n[p^e]}A^e \cong \widehat{A^e}.$$
On the other hand, since the Frobenius action is the same in both
$H^i_{\fm}(A/I)$ and $H^i_{\widehat{\fm}}(\widehat{A}/I
\widehat{A})$, so we may assume that $A$ is a complete regular local
$F$-finite ring.

As $H^i_{\fm}(A/I)$ is an $A[F^e]$-module which is a cofinite
$A$-module, then $\mathcal{D}(H^i_{\fm}(A/I))$ is a finitely
generated unite $A[F^e]$-module (cf. \ref{4.16-4.17}(a)) and
therefore
$$\mathcal{D}(\mathcal{D}(H^i_{\fm}(A/I)))\cong
D(\mathcal{D}(H^i_{\fm}(A/I))).$$ As the functor $D(-)$ transforms
direct limits to inverse limits, then
\begin{eqnarray*}
  D(\mathcal{D}(H^i_{\fm}(A/I)))\!\!\!\! &=&\!\!\!\! D({\vil}(D(H^i_{\fm}(A/I))
  \rightarrow D(F^{e\ast} H^i_{\fm}(A/I)) \rightarrow D(F^{2e\ast} H^i_{\fm}(A/I)) \rightarrow \cdots)  \\
            \!\!\!\! & \cong &\!\!\!\! {\vpl}(\cdots \rightarrow D(D(F^{2e\ast} H^i_{\fm}(A/I)))
  \rightarrow D(D(F^{e\ast} H^i_{\fm}(A/I))) \rightarrow
  D(D(H^i_{\fm}(A/I)))) \\
  \!\!\!\! & \cong &\!\!\!\! {\vpl}_eH^i_{\fm}(A/I^{[p^e]}).
\end{eqnarray*}
As we have seen in the proof of Theorem \ref{unit},
${\vpl}_eH^i_{\fm}(A/I^{[p^e]}) \cong {\vpl}_nH^i_{\fm}(A/I^{n})$.
Therefore,
 $H^i_{\fm}(A/I)$ is $F$-nilpotent if and only if
$\mathcal{D}(H^i_{\fm}(A/I))=0$ (\ref{4.16-4.17}(c)) if and only if
${\vpl}_nH^i_{\fm}(A/I^{n})=0$.
\end{proof}

\begin{remark}
Notice that in the Proposition \ref{f-nil} the $F$-finiteness of $A$
is vital, because it guarantees the $A[F^e]$ structure of the
modules ${\vpl}_nH^i_{\fm}(A/I^{n})$. However, in the light of
\cite[Lemma 4.12]{Bl} if $H^{\dim A-i}_I(A)$ is cofinite, then the
module ${\vpl}_nH^i_{\fm}(A/I^{n})$ is $A[F^e]$-module.
\end{remark}

\begin{cor} \label{Lyu3.2}(\cite[Corollary 3.2]{Lyu06}) Let $(A,\fm)$ be a
 regular local ring and $I$ an ideal of $A$. Then $H^{\dim A-i}_{I}(A)=0$ if
and only if $F^{er}:H^i_{\fm}(A/I) \rightarrow H^i_{\fm}(A/I)$ is
the zero map for some $r>0$.
\end{cor}

\begin{proof} As $\widehat{A}$ is a faithful flat $A$-module then by passing to the completion
we may assume that $A$ is complete regular local ring. Let $H^{\dim
A-i}_{I}(A)=0$, so it is cofinite. Then by the duality
(\ref{duality}) in the introduction, one has
${\vpl}_nH^i_{\fm}(A/I^{n})=0$. Hence, Proposition \ref{f-nil}
implies that $H^i_{\fm}(A/I)$ is $F$-nilpotent.

Conversely, assume that $H^i_{\fm}(A/I)$ is $F$-nilpotent. Then by
Propostion \ref{4.16-4.17}(c) $\mathcal{D}(H^i_{\fm}(A/I))=0$ and
therefore one has $H^{\dim A-i}_{I}(A)=0$. To this end note that,
$\mathcal{D}(H^{\dim A-i}_{\fm}(A/I^{[p^e]})) \cong
\Ext^{i}_A(A/I^{[p^e]},A)$ for all $e \geq 0$ and the Frobenius
powers of $I$ are cofinal with its ordinary powers.
\end{proof}

In the light of Corollary \ref{Lyu3.2}, Lyubeznik \cite{Lyu06}
defined the $\Fdepth$ of a local ring in order to give a solution to
Grothendieck's Problem.

\begin{definition} Let $(A,\fm)$ be a local ring. The $\Fdepth$ of
$A$ is the smallest $i$ such that $F^{er}$ does not send
$H^i_{\fm}(A)$ to zero for any $r$.
\end{definition}

One of elementary properties of $\Fdepth$ shows that $\Fdepth A$ is
equal to the $\Fdepth$ of its $\fm$-adic completion, $\widehat{A}$
(cf. \cite[Proposition 4.4]{Lyu06}) because $H^i_{\fm}(A)\cong
H^i_{\widehat{\fm}}(\widehat{A})$. In the next result we give an
alternative proof of \cite[Lemma 4.2]{Lyu06} to emphasize that
$\Fdepth$ of $A$ is bounded above by its Krull dimension.

\begin{prop} \label{lyu4.2} Let $(A,\fm)$ be a
 local ring and $I$ an ideal of $A$. Then $F^{er}$ does not send $H^{\dim A}_{\fm}(A)$ to zero for any $r$.
 In particular, $0 \leq Fdepth A \leq \dim A$.
\end{prop}

\begin{proof} Since the Frobenius action is the same in both
$H^i_{\fm}(A)$ and $H^i_{\widehat{\fm}}(\widehat{A})$, so we may
assume that $A$ is a complete local ring. Thus, by the Cohen's
Structure Theorem $A \cong R/J$, where $R$ is a complete regular
local ring and $J \subset R$ and ideal. In the contrary, assume that
$H^{\dim A}_{\fm}(R/J)$ is $F$-nilpotent. Then, $\mathcal{D}(H^{\dim
A}_{\fm}(R/J))=0$ (cf. \ref{4.16-4.17}(c)) and with a similar
argument given in the proof of Corollary \ref{Lyu3.2}, one can
deduce the vanishing of $H^{0}_{J}(R)$. Hence, by virtue of
(\ref{duality}), in the introduction, one has ${\vpl}_nH^{\dim A}
_{\fm}(R/J^{n})=0$ which is a contradiction (cf. \cite[Theorem
4.5]{Sch}).
\end{proof}

To investigate the other properties of $\Fdepth$, in the next
Theorem we compare the Frobenius depth of $A$ and $A^{sh}$. For this
reason, let us recall some preliminaries. For a local ring $A$ we
denote by $A^{sh}$ the strict Henselization of $A$. A local ring
$(A,\fm,k)$ is said to be strictly Henselian if and only if every
monic polynomial $f(T)\in A[T]$ for which $\overline{f(T)} \in k[T]$
is separable splits into linear factors in $A[T]$. For more advanced
expositions on this topic we refer the interested reader to \cite
{Mi}.

\begin{prop} \label{Hensel}  Let $(A,\fm)$ be a complete local ring. Then
  $\Fdepth A=\Fdepth A^{sh}$.
\end{prop}

\begin{proof} First assume that $A$ is a regular local ring. We show that $\Fdepth A/I=\Fdepth
(A/I)^{sh}$. Put $\Fdepth A/I=t$. Then, by virtue of Corollary
\ref{Lyu3.2} one has $H^{i}_{I}(A)=0$ for all $i
> \dim A-t$. Due to the faithfully flatness of the inclusion $A
\rightarrow A^{sh}$ and the fact that $A^{sh}$ is a regular local
ring, it implies that $H^{i}_{I}(A^{sh})=0$ for all $i
> \dim A-t$. Again, using Corollary
\ref{Lyu3.2}, it follows that $\Fdepth (A/I)^{sh} \geq t$. To this
end note that $\dim A=\dim A^{sh}$ and $(A/I)^{sh}=A^{sh}/IA^{sh}$.
With the similar argument one has $\Fdepth A/I \geq \Fdepth
(A/I)^{sh}$. This completes the assertion.

Since $A$ is a complete local ring, then by virtue of Cohen's
Structure Theorem, $A$ is a homomorphic image of a regular local
ring $R$, i.e. $A=R/J$ for some ideal $J$ of $R$. Now, we are done
by the previous paragraph. To this end note that
$$\Fdepth A =\Fdepth R/J=\Fdepth (R/J)^{sh}=\Fdepth A^{sh}.$$
\end{proof}

\begin{remark} From Proposition \ref{Hensel} and \cite[Proposition 4.4]{Lyu06} one may
deduce that
$$\Fdepth A=\Fdepth \widehat{((\hat{A})^{sh})},$$
where, $A$ is a local ring.
\end{remark}

In the following, we compare the invariants $\depth$, $\Fdepth$ and
$\fgrade$. Let us recall that the formal grade, $\fgrade(I,R)$, is defined as the index of the minimal nonvanishing
formal cohomology module, i.e., $\fgrade(I,R) = \inf \{i \in \mathbb{Z}|\ \ {\vpl}_nH^i_{\fm}(R/I^n ) \neq 0\}$.

\begin{thrm} \label{comparison}  Let $(A,\fm)$ be a  local $F$-finite ring and let $I$ be
an ideal of $A$. Then
 $$ \fgrade(I,A) \leq \depth A \leq \Fdepth A.$$
\end{thrm}

\begin{proof} We have ${\vpl}_nH^i_{\fm \hat{A}}(\hat{A}/I^n \hat{A}) \cong {\vpl}_nH^i_{\fm}(A/I^n
A)$ (cf. \cite[Proposition 3.3]{Sch}) and $\Fdepth A \cong \Fdepth
\hat{A}$ (cf. \cite[Proposition 4.4]{Lyu06}). As $\fm$-adic
completion of an $F$-finite ring is again $F$-finite, Then, we may
assume that $A$ is a complete local $F$-finite ring. Then $K=A/\fm$
is $F$-finite, then so is every finitely generated algebra over $K$.

By the Cohen's Structure Theorem there exists an $F$-finite regular local ring $(R,\fn)$ with
$A \cong R/J$ where, $J$ is an ideal of $R$.

 Put $\fgrade(J,R)=t$. Then by
definition ${\vpl}_nH^i_{\fn}(R/J^n)=0$ for all $i <t$. It follows
from the Proposition \ref{f-nil} that $H^{i}_{\fn}(R/J)$ is
$F$-nilpotent for all $i<t$, i.e. $\Fdepth R/J \geq t$. With a
similar argument and again using Proposition \ref{f-nil} we have
$\fgrade(J,R)\geq \Fdepth R/J$. Thus, $\fgrade(J,R)= \Fdepth R/J$.

Now, we are done by \cite[Lemma 4.8(b)]{Sch}, \cite[Remark 3.1]{E13}
and the previous paragraph. To this end note that
\begin{eqnarray*}
  \fgrade(I,A) \leq \depth A  &\leq&  \fgrade(J,R)  \\
             &=& \Fdepth R/J \\
             &=& \Fdepth A.
\end{eqnarray*}
\end{proof}

Let $A$ be a complete local ring containing a perfect field $k$ as
its residue field, then $A$ satisfies the condition of Theorem
\ref{comparison}. To this end, note that by Remark \ref{Hart-Speis},
$A$ is F-finite. Furthermore, by virtue of Cohen's Structure Theorem
$A \cong R/J$ for some ideal $J \subset R$, where
$R=k[[x_1,\ldots,x_n]]$ is a regular $F$-finite ring.

\begin{cor} \label{connected}
 Let $(A,\fm)$ be a regular local and $F$-finite ring. Then we have
$$\depth A/I \leq \fgrade(I,A)=\Fdepth A/I \leq \dim A/I.$$
  %In particular, $\fgrade(I,A)=\fgrade(IA^{sh},A^{sh})$.
\end{cor}

\begin{proof}
The assertion follows from what we have seen in the proof of Theorem
\ref{comparison} and
 \cite[Remark 3.1]{E}.
\end{proof}

\begin{remark}
(a) The necessary and sufficient conditions for small values of the
$\Fdepth$ of $A$ is given in \cite[Corollary 4.6]{Lyu06}.
\begin{itemize}
  \item [(1)] $\Fdepth A>0$ if and only if $\dim A>0$.
  \item [(2)] $\Fdepth A>1$ if and only if $\dim A \geq 2$ and the punctured spectrum of $A$ is formally
geometrically connected.
\end{itemize}

Now, let $A$ be $F$-finite and $\depth A=0 < \dim A$.
 Then, one has $\fgrade(I,A)=0 < \Fdepth
 A$. It shows that the inequality in Theorem \ref{comparison}
can be strict.

(b) Keep the assumptions in Corollary \ref{connected}, if $\Fdepth
A/I>1$, then by \cite[Lemma 5.4]{Sch}, one has
$\Supp_{\hat{A}}(\hat{A}/I \hat{A})\setminus \{\hat{\fm}\}$ is
connected. To this end, note that $\hat{A}$ is a local ring so it is
indecomposable.
\end{remark}

\section{Endomorphism rings of local cohomology modules}

Let $R$ be a commutative algebra and $k \subset R$ be a field. We
denote by $\End_k (R)$ the $k$-linear endomorphism ring. The ring of
$k$-linear differential operators $\fD_{R|k} \subseteq \End_k (R)$
generated by the $k$-linear derivations $R \rightarrow R$ and the
multiplications by elements of $R$. By a $\fD_{R|k}$-module we
always mean a left $\fD_{R|k}$-module. The injective ring
homomorphism $R \rightarrow \fD_{R|k}$ that sends $r$ to the map $R
\rightarrow R$ which is the multiplication by $r$, gives $\fD_{R|k}$
a structure of $R$-algebra. Every $\fD_{R|k}$-module $M$ is
automatically an $R$-module via this map. The natural action of
$\fD_{R|k}$ on $R$ makes $R$ a $\fD_{R|k}$-module. If $R=k[[x_1,
\ldots,x_n]]$ is a formal power series ring of $n$ variables $x_1,
\ldots, x_n$ over $k$, then $\fD_{R|k}$ is left and right
Noetherian. Moreover, $\fD_{R|k}$ is a simple ring. Noteworthy, the
local cohomology module $H^{i}_{I}(R),\ i \in \mathbb{Z}$ is a
finitely generated $\fD_{R|k}$-module. For a more advanced
exposition based on differential operators and undefined concepts
the interested reader might consult \cite{Bj}. The similar results
are true whenever $R=k[x_1, \ldots,x_n]$ is a polynomial ring of $n$
variables $x_1, \ldots, x_n$ over $k$ (cf. \cite{Cou}). For a quick
introduction in this topic we refer the reader to \cite{Lyu00}. For
brevity we often write $\fD_R$ for $\fD_{R|k}$ when there is no
ambiguity about the field $k$.

Finally, we bring the following Lemma will be used later.

\begin{lem} \label{B} Let $R$ be a commutative ring containing a field $k$. Let $M$ be both an $R$-module and a
$\fD_R$-module. Then $\Ann_{\fD_R}M =0$ implies $\Ann_R M =0$.
\end{lem}

\begin{proof} Let $r \in \Ann_R M$ be an arbitrary element. As the endomorphism $\varphi_r :R \rightarrow
R$ with $\varphi_r(s) =rs$ for all $s \in R$ is an element of
$\fD_R$ so from $rsM =0$ (for all $s \in R$) we have $\varphi_r(s)M
=0$. That is $\varphi_r(s) \in \Ann_{\fD_R}M =0$, i.e. $rs =
\varphi_r(s) =0$, for all $s \in R$. Hence, we have $r = 0$, as
desired.
\end{proof}

\subsection{Characteristic $0$}

Throughout this subsection $k$ denotes a field of characteristic
zero. Let $R$ be either $k[|x_1, \ldots,x_n|]$, a formal power
series ring or $k[x_1, \ldots,x_n]$ a polynomial ring of $n$
variables $x_1, \ldots, x_n$ over $k$. It is known that the local
cohomology module $H^{i}_{I}(R),\ i \in \mathbb{Z}$ is a holonomic
$\fD_{R}$-module \cite{Lyu93}, i.e. it is a finitely generated
$\fD_{R}$-module with $d(H^{i}_{I}(R))=n$. To this end note that
$d(H^{i}_{I}(R))$ is the so-called Bernstein dimension of
$H^{i}_{I}(R)$ which is by the definition the Krull dimension of the
characteristic variety $\Char (H^{i}_{I}(R))$ (see \cite{Bj} and
\cite{Cou}). In the following result we show that this is not the
case for its endomorphism ring.

For the convenience of the reader let us recall two useful facts.

\begin{fact} \label{Bjork}(cf. \cite{Bj} and \cite{Cou})

\begin{enumerate}
  \item For any nonzero finitely generated $\fD_R$-module $M$, one has $n \leq d(M)
\leq 2n$.
  \item $d(\fD_R)=2n$.
\end{enumerate}
\end{fact}

\begin{prop} \label{A} Let $R$ be as above. Suppose that $H^{i}_{I}(R)\neq 0$, then
 $d(\Hom_{k}(H^{i}_{I}(R),H^{i}_{I}(R)))=2n$. In
  particular $\Hom_{k}(H^{i}_{I}(R),H^{i}_{I}(R))$ is not a
  holonomic $\fD_R$-module.
\end{prop}

\begin{proof} As $H^{i}_{I}(R)$ is a finitely generated $\fD_R$-module, then so is
$\Hom_{k}(H^{i}_{I}(R),H^{i}_{I}(R))$. As $H^{i}_{I}(R)$ is a
non-zero $k$-vector space, it has a non-trivial endomorphism ring.
Consider the homomorphism
\begin{equation}\label{hom}
   \fD_R \stackrel{f}{\longrightarrow} \Hom_{k}(H^{i}_{I}(R),H^{i}_{I}(R))
\end{equation}
of $\fD_R$-modules defined by $f(P)(m) = Pm,\ P \in \fD_R$ and $m
\in H^{i}_{I}(R)$ for all $i \in \mathbb{Z}$. The homomorphism $f$
is injective, as $\fD_R$ is a simple ring. In the light of
(\ref{hom}) and the fact that $d(\fD_R)=2n$, the dimension of
$\Hom_{k}(H^{i}_{I}(R),H^{i}_{I}(R))$ is at least $2n$. Hence, the
assertion follows from Fact (\ref{Bjork}).
\end{proof}

\begin{thrm} \label{annhom}
Let $R$ be either $k[|x_1, \ldots,x_n|]$, a formal power series ring
or $k[x_1, \ldots,x_n]$ a polynomial ring of $n$ variables $x_1,
\ldots, x_n$ over $k$. Suppose that $H^i_I(R)\neq 0$. Then
$$\Ann_{R}(\Hom_{R}(H^{i}_{I}(R),H^{i}_{I}(R)))=0.$$
\end{thrm}

\begin{proof}
As $\Hom_{R}(H^{i}_{I}(R),H^{i}_{I}(R))$ is a non zero $\fD_R$-module and
$\fD_R$ is a simple ring one has
$\Ann_{\fD_R}(\Hom_{R}(H^{i}_{I}(R),H^{i}_{I}(R)))=0$. Now, we are
done by Lemma \ref{B}.
\end{proof}

Now, we may recover a result of Lyubeznik.

\begin{cor} \label{annloc.0} (cf. \cite[Corollary 3.6]{Lyu93})
Let $(R,\fm)$ be a regular local ring containing a field of
characteristic zero. Suppose that $H^i_I(R)\neq 0$. Then
$\Ann_{R}H^{i}_{I}(R)=0$.
\end{cor}

\begin{proof}
Let $\hat{R}$ be the $\fm$-adic completion of $R$. It is known that
$$\Ann_{R}(H^{i}_{I}(R))=
\Ann_{\hat{R}}(H^{i}_{I\hat{R}}(\hat{R}))$$ and $H^i_I (R)\otimes_R
\hat{R} \cong H^i_{I\hat{R}}(\hat{R}) \neq 0$, because of the
faithfully flatness of $\hat{R}$. Then, we may assume that $R$ is
complete. Now we are done by virtue of Cohen's Theorem and Theorem
\ref{annhom}. To this end, let $f \in
\Hom_{R}(H^{i}_{I}(R),H^{i}_{I}(R)) $ be an arbitrary element and $r
\in \Ann_{R}(H^{i}_{I}(R))$ with $rz=0$, for every $z \in
H^{i}_{I}(R)$. Hence, for every $z \in H^{i}_{I}(R)$ one has
$rf(z)=f(rz)=0$ so $\Ann_{R}(H^{i}_{I}(R)) \subseteq
\Ann_{R}(\Hom_{R}(H^{i}_{I}(R),H^{i}_{I}(R)))$.
\end{proof}

\subsection{characteristic $p$}

Throughout this subsection $R$ is a commutative Noetherian ring
containing a field of positive characteristic.

\begin{prop} \label{Blick.Ffinite}
Let $(R,\fm)$ be a $F$-finite regular local ring. Then
$\Hom_{R}(H^{i}_{I}(R),H^{i}_{I}(R))$ is a unite $R[F^e]$-module.
\end{prop}

\begin{proof} As mentioned before (page 5), for a given $i$, $H^{i}_{I}(R)$ is a
unite $R[F^e]$-module. Now we are done by \cite[Corollary 4.10]{Bl}.
\end{proof}

\begin{lem} \label{endo} Let $M$ be an Artinian $R$-module ($R$ is an arbitrary ring with no restriction on its characteristic).
Then there exists an  isomorphism
$$\begin{array}{ll} \ \Phi: \Hom_R (M, M) \rightarrow \Hom_{\widehat{R}} (D(M), D(M)),
\end{array}$$
which sends every $\varphi \in \Hom_R (M, M)$ to $\Phi (\varphi):
\theta \mapsto \theta \circ \varphi$, where $\theta \in D(M)$.

Whenever $R$ is complete and $M$ is a finitely generated $R$-module,
we have again such isomorphism.
\end{lem}

\begin{proof} We prove for the case $M$ is an Artinian module.
For a finite module over a complete local ring the argument is the
same.
 It is clear that the homomorphism $\Phi$ is well-defined. It is enough to show that it is injective and onto.

 ($\Phi$ is injective) Let $0\neq \varphi$ so there exists a
nonzero elemnet $x \in M$ such that $\varphi(x) \neq 0$ in $M$. The
exact sequence $ \ 0 \rightarrow \left\langle  \varphi(x)
\right\rangle \rightarrow M \rightarrow
 (M/\left\langle  \varphi(x) \right\rangle):=L \rightarrow 0$
implies the following short exact sequence $ \ 0 \rightarrow  D(L)
\rightarrow D(M) \rightarrow  D(\left\langle
 \varphi(x) \right\rangle) \rightarrow 0$. It is clear that $D(\left\langle  \varphi(x) \right\rangle) \neq 0$, otherwise $D(L)
\cong D(M)$. Using Matlis duality it implies that $L \cong M$ so
$\varphi(x)=0$, that is contradiction.
   Thus, there exists a nonzero map $\theta: \left\langle
   \varphi(x) \right\rangle \rightarrow E(R/\fm)$. We may extend it to
    the nonzero map $\theta: M \rightarrow E(R/\fm)$ so $\theta \circ \varphi \neq 0$, as required.

 ($\Phi$ is onto) Let $ \psi \in \Hom_{\widehat{R}} (D(M), D(M))$.
We define $\psi$ to send $\theta \in D(M)$ to $\theta \circ
D(\psi)$, where $D(\psi) : D(D(M)) \rightarrow D(D(M))$. As $M$ is
Artinian, $D(D(M)) \cong M$ so $D(\psi): M \rightarrow M$. Then $
\Phi(D(\psi))=\theta \circ D(\psi)= \psi_{\theta}$. Thus, $
\Phi(D(\psi))=\psi$, i.e. $ \Phi$ is onto.
\end{proof}

In the following, the endomorphism ring
$\Hom_{R}(H^{i}_{I}(R),H^{i}_{I}(R))$ is a unite $R[F^e]$-module
without any need for $R$ to be $F$-finite.

\begin{prop}
Let $(R,\fm)$ be a regular local ring. Suppose that $H^{i}_{I}(R)$
is an Artinian $R$-module. Then
$\Hom_{R}(H^{i}_{I}(R),H^{i}_{I}(R))$ is a unite $R[F^e]$-module.
\end{prop}

\begin{proof}
By virtue of Lemma \ref{endo} and the fact that $D(H^{i}_{I}(R))$ is
finitely generated one has
$$F^e(\Hom_{R}(H^{i}_{I}(R),H^{i}_{I}(R))) \cong
\Hom_{R}(F^e(D(H^{i}_{I}(R))),F^e(D(H^{i}_{I}(R))).$$ It follows
from \cite[Lemma 4.12]{Bl} that
$$\Hom_{R}(F^e(D(H^{i}_{I}(R))),F^e(D(H^{i}_{I}(R))) \cong
\Hom_{R}(D(F^e(H^{i}_{I}(R))),D(F^e(H^{i}_{I}(R))).$$ As
$H^{i}_{I}(R)$ is a unite $R[F^e]$-module, then
$$\Hom_{R}(D(F^e(H^{i}_{I}(R))),D(F^e(H^{i}_{I}(R))) \cong
\Hom_{R}(D(H^{i}_{I}(R)),D(H^{i}_{I}(R))).$$ Once again, using Lemma
\ref{endo} the proof is complete.
\end{proof}

Now we are ready to prove the main result of this subsection but
before it we need some preliminaries. It is noteworthy to say that
unite $R[F^e]$-modules have the structure of $\fD$-modules in the
case of characteristic $p>0$ (cf. \cite{Lyu97} and \cite{Bl}).

\begin{definition}
A reduced $F$-finite ring is said to be strongly $F$-regular, if for
all $c \in R$ not in any minimal prime, there exists $q=p^e$ such
that the map $R \rightarrow R^{1/q}$ sending $1 \mapsto c^{1/q}$
splits as an $R$-module homomorphism. Here, $R^{1/q}$ is the
over-ring of $q$th roots of elements in $R$.
\end{definition}

The strongly $F$-regular rings were introduced by Hochster and Huneke
in \cite{Ho-Hu2}. It is known that a regular, $F$-finite ring is
strongly $F$-regular.

\begin{definition}
An $F$-finite ring R has finite F-representation type if there
exists a finite set $\mathcal{S}$ of isomorphism classes of $R$-
modules such that any indecomposable $R$-module summand of
$R^{1/q}$, for any $q = p^e$, is isomorphic to some element of
$\mathcal{S}$.
\end{definition}

Note that an $F$-finite regular ring has finite $F$-representation
type, see \cite{Sm-Va} for more information.

\begin{thrm} \label{charp.endo}
Let $(R,\fm)$ be a complete ring. If $R$ is strongly $F$-regular and
has finite $F$-representation type. Suppose that
$\Hom_{R}(H^{i}_{I}(R),H^{i}_{I}(R))$ is a unite $R[F^e]$-module.
Then $\Ann_R(\Hom_{R}(H^{i}_{I}(R),H^{i}_{I}(R)))=0$.
\end{thrm}

\begin{proof}
As $\Hom_{R}(H^{i}_{I}(R),H^{i}_{I}(R))$ is a unite $R[F^e]$-module,
it then follows from \cite[Proposition 3.6]{Bl} that
$\Hom_{R}(H^{i}_{I}(R),H^{i}_{I}(R))$ has a $\fD_R$-module
structure. On the other hand by \cite[Theorem 4.2.1]{Sm-Va}, $\fD_R$
is a simple ring. Then one has
$$\Ann_{\fD_R}\Hom_{R}(H^{i}_{I}(R),H^{i}_{I}(R))=0.$$ Now we are done
by Lemma \ref{B}.
\end{proof}

As a consequence of Theorem \ref{charp.endo} we can recover a result
of Huneke-Koh.

\begin{cor} \label{annloc.p}(cf. \cite[Lemma 2.2]{Hu-K})
Let $(R,\fm)$ be a regular local ring containing a field $k$ of
characteristic $p>0$. If $H^{i}_{I}(R)\neq 0$,
then $\Ann_R H^{i}_{I}(R)=0$.
\end{cor}

\begin{proof}
Once again similar to what we have seen in the proof of Theorem
\ref{annloc.0} we may assume that $R$ is a complete local ring. Without loss of generality, we may
assume that $k=R/\fm$. Using a suitable  \textit{gonflement} of $R$
there exists a regular local ring $(S,\fn)$ such that $S/\fn$ is the
algebraic closure of $k$ and $(S,\fn)$ is faithful flat over
$(R,\fm)$, see \cite[Chapter IX, Appendice 2]{Bou}. By faithful flatness of $S$ one has $S/\fm S=S/\fn$.
 By passing over $S$ we may assume that $k$ is an algebraically closed field.
Therefore, $k$ is a perfect field. In the light of Remark
\ref{Hart-Speis} we observe that $R$ is $F$-finite. Now the claim
follows from the Proposition \ref{Blick.Ffinite} and Theorem
\ref{charp.endo}.
\end{proof}

\proof[Acknowledgements] I am deeply grateful to the referee for
his/her careful reading of the manuscript, appropriate and constructive suggestions to improve the paper.
 This project is based on a question asked by Professor
Josep \`{A}lvarez-Montaner on the comparison of $\Fdepth$ and
$\fgrade$ when the author was visiting the Universitat
Polit\`{e}cnica de Catalunya, in September 2013. Hereby, I would
like to express my thanks to the Universitat Polit\`{e}cnica and
specially, to Josep, for their support and warm hospitality. I am
also strongly indebted with Josep for several fruitful discussions.
I would like to mentioned that when the author was finishing this
project, he started a joint project with Alberto F. Boix, and gave a
characteristic free proof of the result of Huneke-Koh and Lyubeznik
on the annihilator of local cohomology modules.


\begin{thebibliography}{10}

\bibitem {Bl}
M. Blickle, \emph{The Intersection Homology D-Module in Finite
Characteristic.}, Ph.D. thesis, University of Michigan (2001).

\bibitem {Bj}
J. E. Bj\"{o}rk , \emph{Rings of differential operators.} Amsterdam
North-Holland (1979).

\bibitem {Br-Sh}
M. Brodmann and R.Y. Sharp, \emph{Local cohomology: an algebraic
introduction with geometric applications.} Cambridge Univ. Press,
{\bf 60}, Cambridge, (1998).

\bibitem {Bou}
N. Bourbaki , \emph{\'{E}l\'{e}ments de math\'{e}matique.Alg\'{e}bre
commutative} Chapters 8 et 9, Springer, (2008).

\bibitem {Cou}
S. C. Coutinho, \emph{ A primer of algebraic D-modules}, London
Mathematical Society Student Texts, Cambridge University Press,
(1995).

\bibitem {E13}
M. Eghbali, \emph{On Artinianness of formal local cohomology,
colocalization and coassociated primes.}, Math.
 Scand. {\bf 113}, 5--19 (2013).

\bibitem {E}
M. Eghbali, \emph{On set theoretically and cohomologically complete
intersection ideals.}, Canad. Math. Bull. {\bf 57}(3), 477--484
(2014).

\bibitem {E-Sch}
M. Eghbali and P. Schenzel, \emph{On an Endomorphism Ring of Local
Cohomology.}, Comm. in Alg. {\bf 40}, 4295--4305 (2012).

\bibitem {Gr67}
A. Grothendieck, \emph{Local cohomology.} Lecture notes in
Mathematics, {\bf 41}, Springer, Heidelberg (1967).

\bibitem {Gr68}
A. Grothendieck, \emph{Cohomologie locale des faisceaux et
th\'{e}or\`{e}mes de Lefschetz locaux et globaux (SGA 2).}
North-Holland (1968).

\bibitem {Hart68}
R. Hartshorne, \emph{Cohomological dimension of algebraic
varieties.} Ann. of Math. {\bf 88}(2), 403--450 (1968).

\bibitem {HS}
R. Hartshorne and R. Speiser, \emph{Local cohomological dimension in
characteristic p.} Ann. of Math. {\bf 105}(2), 45--79 (1977).

\bibitem {He-S}
M. Hellus and J. St\"{u}ckrad, \emph{On endomorphism rings of local
cohomology modules.} Proc. Amer. Math. Soc. {\bf 136}(7), 2333--2341
(2008).


\bibitem {Ho-Hu}
M. Hochster and C. Huneke, \emph{Indecomposable canonical modules
and connectedness.}, In: Heinzer, W., Huneke, C., Sally, J., eds.
Commutative Algebra: Syzygies, Multiplicities, and Birational
Algebra. Contemporary Math. Vol. {\bf 159}, pp. 97--208 (1994).

\bibitem {Ho-Hu2}
M. Hochster and C. Huneke, \emph{Tight closure and strong
F-regularity.}, M\'{e}m. Soc. Math. France (N.S.) {\bf 38}, 119--133
(1989).

\bibitem {Hu-K}
C. Huneke and J. Koh, \emph{cofiniteness and vanishing of local
cohomology modules.}, Math. Proc. Cambridge Philoos. Soc. {\bf
110}3, 421--429 (1991).

\bibitem {Ku69}
E. Kunz, \emph{Characterization of regular local rings in
characteristic p.} Amer J. Math. {\bf 91}, 772--784 (1969).

\bibitem {Lyu93}
G. Lyubeznik, \emph{Finiteness properties of local cohomology
modules (an application of D-modules to commutative algebra},
Invent. Math. {\bf 113}1, 41--55 (1993).

\bibitem {Lyu97}
G. Lyubeznik, \emph{F-modules: an application to local cohomology
and D-modules in characteristic $p>0$.} Journal f\"{u}r reine und
angewandte Mathematik {\bf 491}, 65--130 (1997).

\bibitem {Lyu00}
G. Lyubeznik, \emph{Injective dimension of D-modules: a
characteristic-free approach.} J. Pure and Appl. Algebra {\bf 149}
205--212 (2000).

\bibitem {Lyu06}
G. Lyubeznik, \emph{On the vanishing of local cohomology in
characteristic $p>0$.} Composito Math. {\bf 142}, 207--221 (2006).

\bibitem {Ma-Sch}
W. Mahmood and P. Schenzel, \emph{On invariants and endomorphism
rings of certain local cohomology modules.} J. Algebra, {\bf 372},
56--67 (2012).

\bibitem {Mi}
J. S. Milne, \emph{\'{E}tale cohomology.} Princeton Mathematical
Series {\bf 33}, Princeton University Press (1980).

\bibitem {Og}
A. Ogus, \emph{ Local cohomological dimension of Algebraic
Varieties.} Ann. of Math. {\bf 98}(2), 327--365 (1973).

\bibitem {P-S}
C. Peskine and L. Szpiro, \emph{ Dimension projective finie et
cohomologie locale.} Publ. Math. I.H.E.S., {\bf 42}, (1973),
323-395.

\bibitem {Sch}
P. Schenzel, \emph{On formal local cohomology and connectedness.} J.
Algebra, {\bf 315}(2), 894--923 (2007).

\bibitem {Sch09}
P. Schenzel, \emph{On endomorphism rings and dimensions of local
cohomology modules.} Proc. Amer. Math. Soc. {\bf 137}, 1315--1322
(2009).

\bibitem {Sch10}
P. Schenzel, \emph{Matlis duals of local cohomology modules and
their endomorphism rings.} Archiv Math. {\bf 95}, 115--123 (2010).

\bibitem {Sm-Va}
K. E. Smith and M. Van den Bergh, \emph{Simplicity of rings of
differential operators in prime characteristic.}, Proc. London Math.
Soc., {\bf 75}(3), 32--62 (1997).





\end{thebibliography}
\end{document}